\documentclass[a4paper,10pt]{amsart}
\usepackage{amsmath,amsthm,amssymb,amsfonts,enumerate,color}
\newcommand{\norm}[1]{\left\Vert#1\right\Vert}
\newcommand{\abs}[1]{\left\vert#1\right\vert}
\newcommand{\set}[1]{\left\{#1\right\}}

\newcommand{\F}{\mathcal{F}}
\newcommand{\E}{\mathbb{E}}

\newcommand{\B}{\mathbb{B}}

\newcommand{\N}{\mathbb{N}}

\newcommand{\R}{\mathbb{R}}

\newcommand{\I}{\mathcal{I}}

\newcommand{\red}[1]{\textcolor{red}{#1}}

\newtheorem{thm}{Theorem}[section]

\newtheorem{lem}[thm]{Lemma}

\theoremstyle{definition}
\newtheorem{defn}[thm]{Definition}
\newtheorem{rem}[thm]{Remark}

\numberwithin{equation}{section}

\author[C. Zhang]{Chao Zhang}
\address{School of Statistics and Mathematics\\
          Zhejiang Gongshang University \\
          310018 Hangzhou, China}
\email{zaoyangzhangchao@163.com}

\author[W. Chen]{Wei Chen}
\address{School of Mathematics Sciences \\
          Yangzhou University \\
          225002 Yangzhou, China}
\email{weichen@yzu.edu.cn}

\author[P. D. Liu]{Peide Liu}
\address{School of Mathematics and Statistics \\
          Wuhan University \\
          430072 Wuhan, China}
\email{pdliu@whu.edu.cn}

 \thanks{ Supported by the National Natural
Science Foundation of China (No. 11271292 and No. 11101353), Zhejiang Provincial Natural Science Foundation of China (Grant No. LQ14A010001),
the Natural Science Foundation of Jiangsu Education Committee (No. 11KJB110018) and the Natural Science Foundation of Jiangsu Province (No. BK2012682).}



\begin{document}

\title{A characterization of $BMO^\alpha$-martingale spaces by fractional Carleson measures}

\begin{abstract}
We give a characterization of $BMO^\alpha$-martingale
spaces by using fractional Carleson measures.
We get the boudedness of
martingale transform and square function on $BMO^\alpha$-martingale
spaces easily by using this characterization. We also proved the martingale version of Carleson's inequality related with $BMO^\alpha$-martingales.
\end{abstract}

\maketitle

\section{Introduction}

In Harmonic Analysis, the space of functions of bounded mean oscillation, or  $BMO$-space, naturally arises as the class of functions whose deviation from their means over cubes is bounded. We recall that a locally integrable function $f$ on $\R^n$ is in $BMO(\R)$-space if $\displaystyle \norm{f}_{BMO}=\sup_Q\frac{1}{\abs{Q}}\int_Q\abs{f(x)-f_Q}dx,$ where $\abs{Q}$ is the Lebesgue measure of the cube $Q$ in $\R^n$, $f_Q$ denotes the average(or mean value) of $f$ on $Q$, and the supremum is taken over all cubes $Q$ in $\R^n$.  It is a good substitution of the $L^\infty$ space, as a dual space of Hardy space $H^1$, also in the interpolation theory. Also, Carleson measures are among the most important tools in Harmonic Analysis. A positive measure $\mu$ on $\R^{n+1}_+$ is called a Carleson measure if $\displaystyle \norm{\abs{\mu}}=\sup_Q \frac{\mu({Q\times (0,\ l(Q)]})}{\abs{Q}}$, where $l(Q)$ denotes the side length of the cube $Q$. Fefferman and Stein found that the $BMO$-space has a natural and  deep relationship with Carleson measures, see \cite{Fefferman, Grafakos}. In \cite{Wu}, the authors  studied the relationship between the
function in Morry space and a general kind Carleson measure. This general kind Carleson measure is defined as follows. For $0<p<1,$ a positive measure $\mu$ on $\R^{n+1}_+$ is called a bounded $p$-Carleson measure if $\displaystyle \norm{\abs{\mu}}=\sup_Q \frac{\mu({Q\times (0,\ l(Q)]})}{\abs{Q}^{1+p}}$. In this paper we will study an analogue characterization in martingale spaces.

Let $(\Omega, \F, P)$ be a complete probability space. For $1\le
p<\infty$ the usual $L_p$-space of strong $p$-integrable
scalar-valued functions on $(\Omega, \F, P)$ will be denoted by
$L^p(\Omega)$ or simply by $L^p$. Let $\set{\F_n}_{n\ge 0}$ be an
increasing sequence of sub-$\sigma$-fields of $\F$ such that
$\F=\vee \F_n.$ We call a sequence $f=\set{f}_{n\ge 0}$ in $L^1$ to
be a martingale if $\E(f_{n+1}|\F_n)=f_n$ for every $n\ge 0.$ Let
$d_nf=f_n-f_{n-1}$ with the convention that $f_-1=0.$
$\set{d_nf}_{n\ge 0}$ is the martingale difference sequence of $f$.
To avoid unnecessary convergence problem on infinite series we will
assume that all martingales considered in the sequel are finite,
unless explicitly stated otherwise. We will adopt the convention
that a martingale $f=\set{f_n}_{n\ge 0}$ will be identified with its
final value $f_\infty$ whenever the latter exists. And, if $f\in
L^1$ we will denote again by $f$ the associated martingale
$\set{f_n}$ with $f_n=\E(f|\F_n).$ We refer the reader to
\cite{Liu, Long, Weisz} for more information on martingale theory.

The main object of this paper is the $BMO^\alpha$-martingale space
given in the following.

\begin{defn}[$BMO^\alpha$-martingale space]\label{def of BMO alpha}
Let $0\le \alpha\le 1$, $\set{\F_n}_{n\ge 0}$ be an increasing
sequence of $\sigma$-algebras in a probability space $(\Omega, \F,
P)$, $f=\set{f_n}_{n\ge 0}$ an $L^2$-martingale relative to
$\set{\F_n}_{n\ge 0}$. We say that $f$ belongs to $BMO^\alpha$, if
\begin{equation*}
\sup_{n\ge 0}\sup_{A\in \F_n}P(A)^{-1/2-\alpha}\Big(\int_A
|f-f_{n-1}|^2dP\Big)^{1/2}<\infty.
\end{equation*}
The norm in the space $BMO^\alpha$ is $\displaystyle
\norm{f}_{BMO^\alpha}=\sup_{n\ge 0}\sup_{A\in
\F_n}P(A)^{-1/2-\alpha}\Big(\int_A |f-f_{n-1}|^2dP\Big)^{1/2}.$
\end{defn}
Note that $BMO^0$-martingale space is just the $BMO$-martingale
space.

In fact, if we replace the condition in Definition \ref{def of BMO alpha} by
\begin{equation*}
\sup_{n\ge 0}\sup_{A\in \F_n}P(A)^{-1/p-\alpha}\Big(\int_A
|f-f_{n-1}|^pdP\Big)^{1/p}<\infty,
\end{equation*}for any $1\le p <\infty$ and martingales $f\in L^p$-martingale space, the space is equivalent with the space defined in Definition \ref{def of BMO alpha}.

Also, we can use another definition of $BMO^\alpha$-space which is equivalent with Definition \ref{def of BMO alpha}, and we can find it in \cite{Liu, Long} as in the following.

Let $\I^{(n)}$ be the set of all $\F_n$-atoms, $n\ge 0.$ Denote
\begin{equation*}
\omega_n=\sum|I^{(n)}|\chi_{I^{(n)}},
\end{equation*}
where the sum is taken over all $I^{(n)}\in \I^{(n)}$.
\begin{defn}[$BMO^\alpha$-martingale spaces]\label{def of BMO alpha anoth}
Let $0\le \alpha\le 1$, $\set{\F_n}_{n\ge 0}$ be an increasing sequence of $\sigma$-algebras in a probability space $(\Omega, \F, P)$, $f=\set{f_n}_{n\ge 0}$ an $L^2$-martingale relative to $\set{\F_n}_{n\ge 0}$. We say that $f$ belongs to $BMO^\alpha$, if
\begin{equation*}
\sup_n\norm{\omega_n^{-\alpha}E(|f-f_{n-1}|^2|\F_n|)^{1/2}}_\infty<\infty.
\end{equation*}
The norm in the space $BMO^\alpha$-martingale space is
$\norm{f}_{BMO^\alpha}=\sup_n\norm{\omega_n^{-\alpha}E(|f-f_{n-1}|^2|\F_n|)^{1/2}}_\infty.$
\end{defn}

 The classical notion of the general Carleson measure in Harmonic
Analysis has the following martingale analogue.

\begin{defn}[Bounded $\alpha$-Carleson measure]
Let $\mu$ be a nonnegative measure on $\Omega\times \N$, where $\N$
is equipped with the counting measure $dm$. $\mu$ is called a
bounded $\alpha$-Carleson measure$(0\le \alpha<1)$ if $$\norm{\abs{\mu}}_\alpha:=\sup_\tau
\frac{\mu(\hat{\tau})}{P(\tau<\infty)^{1+2\alpha}}<\infty,$$ where
the supremum runs over all stopping times $\tau$ and  $\hat{\tau}$
denotes the ``tent" over $\tau:$ $$\hat{ \tau}=\set{(\omega,
k)\in\Omega\times \N: k\ge \tau(\omega), \tau(\omega)<\infty}.$$
\end{defn}
When $\alpha=0$, the $\alpha$-Carleson measure is just the Carleson
measure in martingale theory. In \cite{Jiao}, the author studied the
relationship between the  Carleson measure and vector-valued
$BMO$-martingale space. And for the scalar-valued case, see
\cite{Long}. In this paper, we will characterize martingales in
$BMO^\alpha$-martingale space in terms of $\alpha$-Carleson
measures.  In fact, we have the following.

\begin{thm}\label{Thm:carleson chara}
The following statements are equivalent:
\begin{enumerate}[(I)]
  \item $f\in BMO^\alpha$;
  \item the measure $\abs{d_kf}^2dP\otimes dm$ is a bounded $\alpha$-Carleson measure, i.e.
  \begin{equation*}
   \sup_\tau \frac{1}{P(\tau<\infty)^{1+2\alpha}}\int_{\widehat{\tau}}
   \abs{d_kf}^2dP\otimes dm<\infty.
  \end{equation*}
\end{enumerate}
\end{thm}

Related with Carleson measures, there is a famous Carleson's inequality in Harmonic Analysis which was first proved by Carleson.
\begin{thm}\label{Thm:Carleson ine}
For any Carleson measure $\mu$ and every $\mu$-measurable function $f$ on $\R^{n+1}_+$ we have
\begin{equation*}
\int_{\R^{n+1}_+}\abs{f(x,t)}^p\ d\mu(x,t)\le C_n \norm{\abs{\mu}}\int_{\R^n}(f^*(x))^p dx
\end{equation*}
for all $0<p<\infty$, where $\displaystyle f^*(x)=\sup_{(y,t)\in \{\abs{y-x}<t\}}\abs{f(x,t)}$.
\end{thm}
We can get a  Carleson's inequality related with $BMO^\alpha$-martingales. Unlike the inequality related with $BMO$-martingales, we can get it only for $1<p<\infty.$

\begin{thm}\label{Thm:carleson inequality}
Let $d\mu=\mu_kdP\otimes dm,$ with $\mu_k$'s being nonnegative random variables, be a bounded $\alpha$-Carleson measure$(0<\alpha<1)$, and $1<p<\infty$. Then for all adapted processes $f=(f_n)_{n\ge 0}$, we have
\begin{equation}\label{carleson ine}
\int_{\Omega\times N}|f_k|^p\mu_k dP\otimes dm\le \frac{p}{p-1}\norm{\abs{\mu}}_\alpha\norm{Mf}_{L^{\frac{1}{2\alpha}}}\norm{Mf}_{L^{p-1}}^{p-1},
\end{equation}
where $Mf$ is the maximal function of $f=(f_n)_{n\ge 0}$.
Conversely, if the above inequality, \eqref{carleson ine}, holds for some $1<p<\infty$, with $\displaystyle \frac{p}{p-1}\norm{\abs{\mu}}_\alpha$  replaced by a constant $C_p$, then $\mu$ is a bounded $\alpha$-Carleson measure and $\displaystyle \norm{\abs{\mu}}_\alpha\le C_p.$
\end{thm}
The paper is organized as follows.  In Section
\ref{Section:Extension}, we give the proofs of Theorem
\ref{Thm:carleson chara} and \ref{Thm:carleson inequality}. In Section \ref{Section:App}, we get the
boundedness of the uniformly bounded martingale transform operator, the square function and the maximal operator on $BMO^\alpha$-martingale spaces easily by using Theorem \ref{Thm:carleson chara}.

Throughout this paper, the letter $C$ will denote a positive
constant which  may change from one instance to another.

\section{Proofs of the main theorems}\label{Section:Extension}

In this section, we will give the proof of Theorem \ref{Thm:carleson
chara}. In order to do this, we need the following lemma.

\begin{lem}\label{Lem:BMO a L2}
Let $f=\set{f_n}_{n\ge 0}$ be an $L^2$-martingale. Then
\begin{equation*}
\norm{f}_{BMO^\alpha}=\sup_\tau
P(\tau<\infty)^{-1/2-\alpha}\norm{f-f_{\tau-1}}_{L^2},
\end{equation*}
where the supremum  is taken over all stopping times $\tau.$
\end{lem}
\begin{proof}
Assume that  $\norm{f}_{BMO^\alpha}<\infty$ and $\tau$ is any
stopping time. Then
\begin{align*}
P(\tau<\infty)^{-1-2\alpha}\norm{f-f_{\tau-1}}_{L^2}^2
&=P(\tau<\infty)^{-1-2\alpha}\sum_{n=1}^\infty\int_{\set{\tau=n}}|f-f_{n-1}|^2dP\\
&\le \norm{f}_{BMO^\alpha}^2P(\tau<\infty)^{-1-2\alpha}\sum_{n=1}^\infty
P(\tau=n)^{1+2\alpha}=\norm{f}_{BMO^\alpha}^2.
\end{align*}

Conversely, let $\beta=\sup_\tau
P(\tau<\infty)^{-1/2-\alpha}\norm{f-f_{\tau-1}}_{L^2}$. For any
$n\ge 1$ and $A\in \F_n$, define
\begin{equation*}
\tau_A(\omega)=
\begin{cases}
n,\quad\ \   \omega\in A, \\
\infty,\ \ \ \ \omega\notin A.
\end{cases}
\end{equation*}
Then,
\begin{align*}
P(A)^{-1-2\alpha}\int_A
\abs{f-f_{n-1}}^2dP=P(\tau_A<\infty)^{-1-2\alpha}\norm{f-f_{\tau_A-1}}_{L^2}^2\le
\beta^2.
\end{align*}
This proves that $\norm{f}_{BMO^\alpha}\le \beta.$ We complete the
proof of the lemma.
\end{proof}

With Lemma \ref{Lem:BMO a L2} above, we can give the proof of Theorem \ref{Thm:carleson chara} as follows.

\begin{proof}[Proof of Theorem \ref{Thm:carleson chara}]
$(I)\Rightarrow (II).$ Assume that $f\in BMO^\alpha$.  For any $1\le
n\le m$ and $A\in \F_n,$ we have 
\begin{multline*}
\int_A \sum_{k=n}^{m}\abs{d_kf}^2dP \le
C\int_A\abs{f_m-f_{n-1}}^2dP\le C\int_A\abs{f-f_{n-1}}^2dP\le
C P(A)^{1+2\alpha}\norm{f}_{BMO^\alpha}^2.
\end{multline*}
This means
\begin{equation*}
P(A)^{-1-2\alpha}\int_A \sum_{k=n}^{\infty}\abs{d_kf}^2dP\le
C\norm{f}_{BMO^\alpha}^2, \quad \text{for any } n\ge 1 \text{ and }
A\in \F_n.
\end{equation*}
We define \begin{equation*}
\tau_A(\omega)=
\begin{cases}
n,\quad\ \   \omega\in A, \\
\infty,\ \ \ \ \omega\notin A.
\end{cases}
\end{equation*}
Then, since $\Omega\in \F_\tau$ always, we
get
\begin{align*}
\frac{1}{P(\tau_A<\infty)^{1+2\alpha}}\int_{\widehat{\tau_A}}\abs{d_kf}^2dP\otimes
dm
&=\frac{1}{P(\tau_A<\infty)^{1+2\alpha}}\int_\Omega\sum_{k=\tau_A}^\infty
\abs{d_kf}^2\chi_{\set{\tau_A<\infty}}dP \\
&=\frac{1}{P(A)^{1+2\alpha}}\int_A\sum_{k=n}^\infty
\abs{d_kf}^2 dP\le C\norm{f}^2_{BMO^\alpha}.
\end{align*}
By taking supremum over all stopping times, we know that the measure $\abs{d_kf}^2dP\otimes dm$ is a bounded $\alpha$-Carleson measure.

$(II)\Rightarrow (I).$ Let us consider the new $\sigma$-fields
$\set{\F_{k\vee\tau}}_{k\ge 1}$ and the corresponding martingale
$\tilde{f}$ generated by $f-f_\tau$. Then by Doob's stopping time
theorem,
$$\tilde{f}_k=\E\left(f-f_\tau|\F_{k\vee \tau}\right)
=\E\left(f|\F_{k\vee \tau}\right)-f_\tau=f_{k\vee \tau}-f_\tau.$$ By
Burkholder-Gundy's inequality, we get
\begin{multline*}
\norm{f-f_\tau}_{L^2}^2=\norm{\tilde{f}}_{L^2}^2 \le
C\norm{\left(\sum_{k=1}^\infty
\abs{d_k\tilde{f}}^2\right)^{1/2}}_{L^2}^2
=C\norm{\left(\sum_{k=1}^\infty \abs{f_{(k+1)\vee \tau}-f_{k\vee \tau}}^2\right)^{1/2}}_{L^2}^2\\
=C\norm{\left(\sum_{k=\tau}^\infty
\abs{f_{k+1}-f_{k}}^2\right)^{1/2}}_{L^2}^2=C\norm{\left(\sum_{k=\tau+1}^\infty
\abs{d_kf}^2\right)^{1/2}\chi_{\set{\tau<\infty}}}_{L^2}^2.
\end{multline*}
Therefore,
\begin{align}\label{est inte f tau}
\norm{f-f_{\tau-1}}_{L^2}^2 &\le C\left(\norm{f-f_{\tau}}_{L^2}^2
+\norm{f_\tau-f_{\tau-1}}_{L^2}^2\right)\nonumber\\
&=C\norm{\left(\sum_{k=\tau}^\infty
\abs{d_kf}^2\right)^{1/2}\chi_{\set{\tau<\infty}}}_{L^2}^2.
\end{align}
Hence, by Lemma \ref{Lem:BMO a L2} and (\ref{est inte f tau})  we obtain
\begin{multline*}
\norm{f}_{BMO^\alpha}\le C\sup_\tau P(\tau<\infty)^{-1/2-\alpha}
\norm{\left(\sum_{k=\tau}^\infty \abs{d_kf}^2\right)^{1/2}\chi_{\set{\tau<\infty}}}_{L^2}\\
=C\sup_\tau
\left(\frac{1}{P(\tau<\infty)^{1+2\alpha}}\int_{\widehat{\tau}}\abs{d_kf}^2dP\otimes
dm\right)^{1/2}<\infty.
\end{multline*}

By the proof above, we know that  $\displaystyle
\norm{f}_{BMO^\alpha}\sim \sup_\tau
\left(\frac{1}{P(\tau<\infty)^{1+2\alpha}}\int_{\widehat{\tau}}
\abs{d_kf}^2dP\otimes dm\right)^{1/2}$.
\end{proof}

In the following, we will give the proof of Theorem \ref{Thm:carleson inequality}.

\begin{proof}[Proof of Theorem \ref{Thm:carleson inequality}]
First, assume that $\mu$ is a bounded $\alpha$-Carleson measure, $f=(f_n)_{n\ge 0}$ is adapted. For any $\lambda>0,$ define the stopping time
$\tau=\inf\{n:|f_n|>\lambda\}.$ Then we have
\begin{equation*}
Mf(\omega)=\chi_{\{\tau<\infty\}}(\omega)\ \text{and}\ \{(\omega, k):|f_k|>\lambda\}\subset\{(\omega, k):k\ge \tau(\omega),\  \tau(\omega)<\infty\}.
\end{equation*}
Hence, \begin{align*}
\int_{\Omega\times N}|f_k|^p\mu_k dP\otimes dm&=p\int_0^\infty \lambda^{p-1}\mu({\{(\omega,k):\abs{f_k}>\lambda\}})d\lambda\\
&\le p\int_0^\infty \lambda^{p-1}\mu(\{(\omega, k):k\ge \tau(\omega),\ \tau(\omega)<\infty\})d\lambda\\
&\le \norm{\abs{\mu}}_\alpha p\int_0^\infty \lambda^{p-1}P(\tau<\infty)^{1+2\alpha}d\lambda\\
&=\norm{\abs{\mu}}_\alpha p\int_0^\infty \lambda (P(\{Mf>\lambda\}))^{2\alpha}\ \lambda^{p-2}P(\{Mf>\lambda\})d\lambda\\
&\le \norm{\abs{\mu}}_\alpha \norm{Mf}_{L^{\frac{1}{2\alpha},\infty}} \cdot p\int_0^\infty  \lambda^{p-2}P(\{Mf>\lambda\})d\lambda\\
&\le \frac{p}{p-1}\norm{\abs{\mu}}_\alpha \norm{Mf}_{L^{\frac{1}{2\alpha}}}\norm{Mf}_{L^{p-1}}.
\end{align*}
Thus we proved the inequality \eqref{carleson ine}.
Reversely, we assume that \eqref{carleson ine} holds for some $1<p<\infty$ with a constant $C_p$. For any stopping time $\tau,$ let $f_n(\omega)=\chi_{\{\tau\le n\}}(\omega),$ for any $\omega \in \Omega,\ n\ge 0$. Then $Mf(\omega)=\chi_{\{\tau<\infty\}}(\omega),$ and
\begin{align*}
&\int_{\Omega\times N}\chi_{\{\tau\le k\}}\mu_k dP\otimes dm=\int_{\Omega\times N}|f_k|^p\mu_k dP\otimes dm\le C_p\norm{Mf}_{L^{\frac{1}{2\alpha}}}\norm{Mf}_{L^{p-1}}^{p-1}=C_p \abs{\{\tau<\infty\}}^{1+2\alpha}.
\end{align*}
This means that $\displaystyle \norm{\abs{\mu}}_\alpha\le C_p$. We complete the proof.
\end{proof}
\begin{rem}
In particular, if $d\mu=\abs{d_kf}^2 dP\otimes dm$ with $f=(f_n)_{n\ge 0}$ is a martingale, then, combined with Theorem \ref{Thm:carleson chara}, we have that $f\in BMO^\alpha,$ if and only if
\begin{align*}
\int_{\Omega\times N}\abs{f_k}^p \abs{d_kf}^2 dP\otimes dm=\int_{\Omega\times N}\abs{f_k}^p d\mu\le C_p\norm{Mf}_{L^{\frac{1}{2\alpha}}}\norm{Mf}_{L^{p-1}}^{p-1},
\end{align*}for all adapted $f=(f_n)_{n\ge 0}$ and some $1<p<\infty.$
\end{rem}

\section{Applications}\label{Section:App}
In this section, we will give some applications of our main theorem.
In order to adapt our result to the applications we need the
following remark.

\begin{rem}\label{Remark:vector charac}
Theorem \ref{Thm:carleson chara} can also be stated in a
Hilbert-valued setting. If $f$ takes value in a Hilbert space
$\mathbb{H}$ and  the absolute values in the statements are replaced
by the norm in $\mathbb{H}$, then the result holds also.
\end{rem}

For more information about the vector-valued martingales, we refer the reader to \cite{Liu, Martinez, Torrea, MarTorrea}.

\subsection{The boundedness of martingale transform on $BMO^\alpha$-martingale spaces}
\begin{thm}
Let $\set{v_n}_{n\ge 0}$ be a uniformly bounded $\F_n$-predictable
sequence. Then the martingale transform given by
$$(Tf)_n=\sum_{k=0}^nv_kd_kf$$ is bounded on $BMO^\alpha$-martingale
spaces.
\end{thm}

\begin{proof}
For any $f\in BMO^\alpha,$  by Theorem \ref{Thm:carleson chara}, we
have that for any stopping time $\tau,$
$$\frac{1}{P(\tau<\infty)^{1+2\alpha}}\int_{\widehat{\tau}}\abs{d_kf}^2dP\otimes dm\le C\norm{f}_{BMO^\alpha}^2.$$
Since $\set{v_k}_{k\ge 0}$ is uniformly bounded, then for any
stopping time $\tau,$
\begin{multline*}
\frac{1}{P(\tau<\infty)^{1+2\alpha}}\int_{\widehat{\tau}}\abs{d_k(Tf)}^2dP\otimes
dm=\frac{1}{P(\tau<\infty)^{1+2\alpha}}\int_{\widehat{\tau}}\abs{v_kd_kf}^2dP\otimes
dm \\ \le C\cdot
\frac{1}{P(\tau<\infty)^{1+2\alpha}}\int_{\widehat{\tau}}\abs{d_kf}^2dP\otimes
dm\le C\norm{f}_{BMO^\alpha}^2.
\end{multline*}
Then by Theorem \ref{Thm:carleson chara}, we get $Tf\in BMO^\alpha$
and $\norm{Tf}_{BMO^\alpha}\le C\norm{f}_{BMO^\alpha}$.
\end{proof}

\subsection{The boundedness of square function on $BMO^\alpha$-martingale spaces}

\begin{thm}
The square function $\displaystyle S(f)=\left(\sum_{k=0}^\infty
|d_kf|^2\right)^{1/2}$ is bounded on $BMO^\alpha$-martingale spaces.
\end{thm}
\begin{proof}
Let us consider the $\ell^2$-valued martingale transform:
$$Uf=\sum_{k=1}^\infty v_kd_kf\quad  \hbox{and}\quad U_nf=\sum_{k=1}^nv_k d_kf$$
with $v_k=(0,\cdots,0,1,0,\cdots)$ for any $k\ge 1.$ It maps a
$\R$-valued martingale into an $\ell^2$-valued martingale. So, for
any $f\in BMO^\alpha,$ by Theorem \ref{Thm:carleson chara},
\begin{align*}
\frac{1}{P(\tau<\infty)^{1+2\alpha}}\int_{\widehat{\tau}}
\norm{d_k(Uf)}_{\ell^2}^2dP\otimes dm
=\frac{1}{P(\tau<\infty)^{1+2\alpha}}\int_{\widehat{\tau}}\abs{d_kf}^2dP\otimes
dm\le C\norm{f}_{BMO^\alpha}^2.
\end{align*}
Therefore, by Remark \ref{Remark:vector charac}, we have $Uf\in
BMO_{\ell^2}^\alpha$ and $\norm{U(f)}_{BMO_{\ell^2}^\alpha}\le C\norm{f}_{BMO^\alpha}.$ Since, for any $n\ge 1,$
$$\norm{U_nf}_{\ell^2}=\left(\sum_{k=1}^n\abs{d_k f}^2\right)^{1/2}=S_n(f),$$
we have, for any $A\in \F_n,$
\begin{multline*}
P(A)^{-1-2\alpha}\int_A\abs{S(f)-S_{n-1}(f)}^2dP
=P(A)^{-1-2\alpha}\int_A\abs{\norm{Uf}_{\ell^2}-\norm{U_{n-1}(f)}_{\ell^2}}^2dP\\
\le P(A)^{-1-2\alpha}\int_A\norm{U(f)-U_{n-1}(f)}_{\ell^2}^2dP\le
C\norm{U(f)}^2_{BMO_{\ell^2}^\alpha}.
\end{multline*}
Whence, $\norm{S(f)}_{BMO^\alpha}\le
C\norm{U(f)}_{BMO_{\ell^2}^\alpha}\le C\norm{f}_{BMO^\alpha}$.
\end{proof}

\subsection{The boundedness of maximal function on $BMO^\alpha$-martingale spaces}

\begin{thm}\label{Thm:maximal}
The maximal function $\displaystyle M(f)=\sup_{n\ge 0}|f_n|$ is bounded on $BMO^\alpha$-martingale spaces.
\end{thm}

The proof of Theorem \ref{Thm:maximal} is similar to the proof of the boundedness of maximal function on $BMO$-martingale spaces with small variations as in \cite{Long}.

\subsection{UMD Banach lattice}

In \cite{Burkholder 1981}, Burkholder considered the UMD spaces as in the following.
\begin{defn}
A Banach space $\B$ is said to be a UMD space if for all $\B$-valued
martingales $\set{f_n}_{n\ge 0}$ and all sequences
$\set{\varepsilon_n}_{n\ge 0}$ with $\varepsilon_n=\pm 1,$
$$\norm{\varepsilon_0 d_0f+\cdots+\varepsilon_k d_kf}_{L^p_\B}
\le C_p\norm{d_0f+\cdots+d_kf}_{L^p_{\B}}, \quad \hbox{for all }
k\ge 0,$$ where $1<p<\infty$ and the constant $C_p>0$ is independent
of $k.$
\end{defn}

It is known that the existence of one $p_0$ satisfying the
inequality is enough to assure the existence of the rest of $p,\
1<p<\infty.$ In this section, we concentrate on the Banach lattice.
Let $\B$ be a Banach lattice. Without loss of generality we assume
that $\B$ is a Banach lattice of measurable functions on measure
space $(\Omega, dP)$. We refer the reader to \cite{LindenTza} for
more information on Banach lattices.

By Theorem 3.4 in \cite{Jiao} and our characterization theorem of $BMO^\alpha$-martingales in  Theorem \ref{Thm:carleson chara},  we can get a characterization of UMD Banach lattice in the following theorem.

\begin{thm}\label{Thm:UMD by Carleson}
Let $\B$ be a Banach lattice. The following statements are equivalent:
\begin{enumerate}[(I)]
\item $\B$ is a \textup{UMD} space;
\item there exists a constant $C>0$ such that
$$
C^{-1}\norm{f}_{BMO^\alpha_\B}\le \sup_{\tau}
\left(\frac{1}{P(\tau<\infty)^{1+2\alpha}}\int_{\widehat{\tau}}\norm{d_kf}_{\B}^2dP\otimes
dm\right)^{1/2}\le C\norm{f}_{BMO^\alpha_\B},
$$
for any $\B$-valued martingale $f.$
\end{enumerate}
\end{thm}
\end{document}